\documentclass[12pt]{amsart}
\advance\hoffset by -0.5 truecm
\usepackage{amssymb}
\usepackage{graphicx}
\usepackage{subfigure}
\newtheorem{Theorem}{Theorem}[section]
\newtheorem{Lemma}[Theorem]{Lemma}

\def\V{\mbox{Var}}

\def\R\re
\def\V{\bf V}

\def \re{{\mathbb R}}

\def \0{\lambda_{0}}

\begin{document}
\title[Isoperimetric profile comparisons and Yamabe constants]{Isoperimetric profile comparisons and Yamabe constants}

\author[J. Petean]{Jimmy Petean}\thanks{J. Petean is supported
by grant 106923-F of CONACYT}
 \address{CIMAT  \\
          A.P. 402, 36000 \\
          Guanajuato. Gto. \\
          M\'exico \\
           and Departamento de Matem\'{a}ticas, FCEyN \\
          Universidad de Buenos Aires, Argentina.}
\email{jimmy@cimat.mx}

\author[J. Ruiz]{Juan Miguel Ruiz}
 \address{CIMAT  \\
          A.P. 402, 36000 \\
          Guanajuato. Gto. \\
          M\'exico }
\email{miguel@cimat.mx}


\date{}


\begin{abstract}  We estimate from below the isoperimetric profile of $S^2 \times \re^2$ and 
use this information to obtain lower bounds for the Yamabe constant of
$S^2 \times \re^2$. This provides a lower bound for the Yamabe invariants
of products $S^2 \times M^2$ for any closed Riemann surface $M$. Explicitly
we show that $Y(S^2 \times M^2) > (2/3) Y(S^4 )$.

\end{abstract}

\maketitle

\section{Introduction} 
Given a conformal class  $[g]$ of Riemannian metrics on a closed manifold $M^n$ the 
{\it Yamabe constant} of $[g]$ is defined as the infimum
of the (normalized) total scalar curvature functional restricted to $[g]$:

$$Y(M,[g])  = \inf_{h\in [g]} \ \frac{ \int_M  s_h \  dvol_h }{ Vol(M,h)^{\frac{n-2}{n}}},$$

\noindent
where $s_h$ and $dvol_h$ are the scalar curvature and volume element of $h$.

If we express metrics in the conformal class of $g$ as $f^{4/(n-2)} \ g$ then we obtain the expression 

$$Y(M,[g]) = \inf_f \frac{ a_n  \int_M {\| \nabla f \|}^2 dvol(g) + \int_M s_g f^2 dvol(g) }{ (\int_M f^p dvol(g))^{2/p} }
=  \inf_{f\in L_1^2 (M)}  Y_g (f).$$

\noindent
Here we let $p=p_n = 2n/(n-2)$ and we will call $Y_g$ the Yamabe functional (corresponding to $g$).

If $f$ is a critical point of $Y_g$ then the corresponding metric $f^{4/(n-2)} \ g$ has
constant scalar curvature.  
H. Yamabe  introduced these notions in \cite{Yamabe} and gave a proof
that $Y(M,[g])$ is always achieved. His proof contained a mistake which was corrected
in a series of steps N. Trudinger \cite{Trudinger}, T. Aubin \cite{Aubin} and R. Schoen 
\cite{Schoen}, proving in this way the existence of at least one metric of constant
scalar curvature  in $[g]$. 

Later on O. Kobayashi in \cite{Kobayashi} and R. Schoen in \cite{Schoen2} introduced what we will call 
the {\it Yamabe
invariant} of $M$, $Y(M)$, 
as the supremum of the Yamabe constants of all conformal classes of Riemannian
metrics on $M$:

$$Y(M) = \sup_{ \{ [g] \} } Y(M,[g]) .$$


By a local argument  T. Aubin showed in \cite{Aubin} that the Yamabe constant
of any conformal class of metric on any $n$-dimensional manifold is bounded above
by $Y(S^n ,[g^n_0 ] )$, where by $g^n_0$ we will denote from now on the round metric of
sectional curvature one on $S^n$. It follows that $Y(S^n) =Y(S^n , [g^n_0 ])$ and for 
any $n$-dimensional manifold $M$, $Y(M) \leq Y(S^n )$. 


In this article we will be concerned with the problem of finding lower bounds for the Yamabe
constants of particular conformal classes. 
If the Yamabe constant of a conformal class $[g]$ is non-positive one has a good lower bound
$Y(M,[g]) \geq \inf_M (s_g ) Vol(M,g)^{2/n}$, as pointed out by O. Kobayashi \cite{Kobayashi2}. 
There is no similar lower
bound when the Yamabe constant is positive, and this is one explanation why the positive case is 
much more difficult to study. 
For instance one can use Kobayashi's lower bound to prove that if $M^n$ is a closed
$n$-manifold and $\overline{M}$ is obtained by performing surgery on a sphere of dimension 
 $ k\leq n-3$ then $Y(\overline{M} ) \geq Y(M)$ \cite{Yun}.  Certain computations of the invariant 
can be deduced from this result, for instance in dimension 4 it implies that if $Y(M) \leq 0$ then
$Y(M\# (S^1\times S^3 ))=Y(M)$ \cite{Petean}.  But for the above reasons studying the behavior
of the invariant under surgery in the positive case becomes much more difficult and it is still
unknown if the surgery result holds as in the non-positive case. Recently B. Ammann, M. Dahl
and E. Humbert \cite{Ammann} proved that there is a positive constant $\lambda_{n,k}$, 
which depends only on $n$ and $k$, such that 
$Y(\overline{M} )\geq \min \{ Y(M), \lambda_{n,k} \}$.

There is a  good lower bound for the Yamabe constant of the conformal class of
a metric of positive Ricci curvature as proved by S. Ilias in \cite{Ilias}: if $Ricci(g) \geq k g$, then
$Y(M,[g])\geq (n-1) k Vol(M,g)^{2/n}$. To
obtain  this  lower bound S. Ilias  compares $Y_g (f)$ with $Y_{g^n_0} (f_* )$, where 
$f$ is any smooth positive function in $M$ and $f_*$ is the spherical
symmetrization of $f$ (as explained below).  The comparison of the $L^2$ and
$L^p$ norms of the functions is immediate  
and to compare the $L^2$-norms of the gradients one applies the coarea formula 
and the comparison of the isoperimetric
profiles given by the Levy-Gromov isoperimetric inequality. The same type of
argument works as soon as one has lower bounds for the isoperimetric profile and
the scalar curvature of a metric $g$, and this is the idea we will apply in this work.

We will denote by $I_{(M,g)} : (0,Vol(M,g))\rightarrow \re_{\geq 0}$ the isoperimetric profile of
$(M,g)$. Namely, for any $t>0$ we consider all the regions in $M$ of volume $t$ and let
$I_{(M,g)} (t)$ be the infimum of the volumes of their boundaries. If a region $U$ realizes
the infimum it is called an {\it isoperimetric region} and $\partial U$ is called an {\it isoperimetric 
hypersurface}.  For all manifolds appearing in this article isoperimetric regions are known
to exit and their boundaries are smooth hypersurfaces. Note that sometimes in the definition
of the isoperimetric profile there is a normalization by the volume of the manifold, but since
we will be interested in Riemannian manifolds with infinite volume this is not possible.

In  this article we will concentrate in obtaining a lower bound for  the 
Yamabe constant of $S^2 \times \re^2$. 
First we point out that for a non-compact  manifold $(W^n ,g)$ of  positive scalar curvature
we define its Yamabe constant by

$$Y(W,g) = \inf_{h\in L_1^2 (W)} \frac{a_n  \int_W  | \nabla h |^2 dvol(g) + \int_W s_g h^2 dvol(g) }{ (\int_W h^p dvol(g))^{2/p} }
=  \inf_{h\in L_1^2 (W)}  Y_g (h).$$

To apply the ideas mentioned above we need estimates for the isoperimetric profile. 
But the  isoperimetric regions  in  $S^k \times \re^n$ are not known when $n\geq 2$.
The isoperimetric profiles of the cylinders $S^k \times \re$ were described by
   by R. Pedroza in
\cite{Pedro}. In Section 2 we will use his results   and the Ros Product Theorem in \cite{Ros} to prove the
following comparison:

\begin{Theorem}  $I_{(S^2 \times \re^2 ,g^2_0 +dx^2)}  \geq \frac{2 \sqrt{\epsilon}}{12^{3/8}} I_{(S^4 , 2^{3/2}  3^{1/4}\epsilon g_0^4 )}$, where
$\epsilon = (1.047)^2$.

\end{Theorem}

Given a non-negative smooth function $f \in L^1_2 (S^2 \times \re^2 )$ we will
build in Section 3 symmetrizations $f_* $ which are nonincreasing radial
function on the sphere $S^4$ and by using the previous theorem and the
ideas mentioned above we will prove:

\begin{Theorem} $Y(S^2 \times \re^2 , [g_0 +dx^2 ]) \geq \frac{\sqrt{2}\epsilon}{3^{3/4}} Y(S^4 )$.
\end{Theorem}

Note that 

$$\frac{\sqrt{2}\epsilon}{3^{3/4}}  =\left(  \frac{2 \sqrt{\epsilon}}{12^{3/8}} \right)^2 \approx 0.68 .$$

Similar ideas can be applied to the products $S^k \times \re^n$, for any $k$ and $n$. But
since some non-trivial numerical computation must be carried on it seemed better 
to focus in the particular case of $S^2 \times \re^2$. 

Now, since for any Riemannian metric $g$ on any 2-dimensional closed manifold
$M$ it is proven in \cite[Theorem 1.1]{Akutagawa} that

$$\lim_{r\rightarrow \infty} Y(S^2 \times M, [g_0 + rg] )=Y(S^2 \times \re^2 , [g_0 +dx^2 ]) ,$$

\noindent
we obtain as a corollary that

\begin{Theorem} If $M$ is a  closed 2-dimensional manifold then $Y(S^2 \times M)\geq
\frac{\sqrt{2} \epsilon }{3^{3/4}} Y(S^4 )$.
\end{Theorem}

As far as we know this is the best result known about the Yamabe invariants of $S^2
\times M^2$ when $M$ is a Riemann surface of genus $g\geq 1$. $S^2 \times M$ admits
metrics of positive scalar curvature, and so it is known that $Y(S^2 \times M) \in (0,Y(S^4 )]$.
In \cite{Petean2} it is proved that $Y(S^2 \times M) >0.0006 Y(S^4 ) $
(see \cite[Theorem 2]{Petean2} for the explicit constant). But no other estimate is known
to the best of the author's knowledge. In the case of $S^2 \times S^2$ the product Einstein metric, $g_E$,
is a Yamabe metric (since it is the only unit volume metric of constant scalar curvature
in its conformal class by the classical theorem of M. Obata \cite{Obata}). Then one knows that
$Y(S^2 \times S^2 ) \geq Y(S^2 \times S^2 , [g_E]) = 16 \pi \approx 0.816 Y(S^4 )$ and it was
proved by C. B\"{o}hm, M. Wang  and W. Ziller in \cite{Bohm} that the inequality is strict. So
in particular the last theorem does not give any new information for this case. We also point
out that there are a few computations where the invariant falls into the interval
$(0,Y(S^n ))$, only in dimensions 3 and 4. In dimension 3 it was proved by H. Bray and
A. Neves \cite{Bray} that the conformal class of the constant curvature metric 
on the projective space ${\bf RP^3}$ achieves
the Yamabe invariant and so $Y({\bf RP^3}) = 2^{-2/3} Y(S^3 ) \equiv 0.63 \  Y(S^3 )$ and it
was later shown by K. Akutagawa and A. Neves that this value is also the 
Yamabe invariant of the connected sum of ${\bf RP^3}$ with any number of copies of
$S^2 \times S^1$ \cite{Neves}. In dimension 4 C. LeBrun proved that  the conformal class of the
Fubini-Study metric on ${\bf CP^2}$ realizes the Yamabe invariant of 
${\bf CP^2}$ and so $Y({\bf CP^2}) = 12 \sqrt{2} \pi \approx 0.87 \ Y(S^4 )$ and  
later M. Gursky and C. LeBrun showed that this is also the value of
the Yamabe invariant of the connected sum
of ${\bf CP^2}$ with any number of copies of $S^3 \times S^1$ \cite{Gursky}.


\section{Estimating the isoperimetric profile of $S^2 \times \re^2$}

In this section we will prove Theorem 1.1. 
First
we discuss a comparison between the isoperimetric profiles of $(S^2\times \mathbf{R},g^2_0+dt^2)$ and $(S^3,\lambda_1 g^3_0 )$, and between those of $(S^3\times \mathbf{R},\lambda_1 g^3_0 +dt^2)$ and $(S^4,\lambda_2 g^4_0)$, where $g^n_0$ is the round metric for $S^n$, and $\lambda_1=2$, $\lambda_2= (2)^{3/2}(3)^{(1/4)} \epsilon$,
(where $\epsilon =(1.047)^2$). 
We picked $\lambda_1 =2$  to match the maximums of the isoperimetric profiles  $I_{(S^3,\lambda_1 g^3_0)}$ and $I_{(S^2\times \mathbf{R},g_0^2+dt^2)}$ and we will
prove in Subsection 2.1 that  $I(S^2\times \re ,g_2+dt^2) \geq I_{(S^3,2 g^3_0)} $. To obtain our lower bound on the Yamabe constant of $S^2 \times \re^2$
we will need to compare the isoperimetric profile of $S^2 \times \re^2$ with one of a 4-sphere $\lambda_2 g^4_0$. As we increase $\lambda_2$ the scalar
curvature decreases and this improves the lower bound. But the isoperimetric profile also increases and this makes our lower bound smaller. The value
$\lambda_2= (2)^{3/2}(3)^{(1/4)} \epsilon$ is the one for which one obtains the best lower bound for the Yamabe constant. This should become clear in Section 3.

We will denote by $S^n (k)$ the round n-sphere of scalar curvature $k$. Note that
according to this notation $(S^3, 2 g_3 )=S^3 (3)$ and $(S^4 , (2)^{3/2}(3)^{(1/4)} \epsilon g^4_0 ) = S^4 (12^{3/4}/(2\epsilon ))$.

The isoperimetric profile for the spherical cylinder $(S^n\times \re ,g^n_0+dt^2)$, $n\geq 2$, was recently studied by R. Pedrosa \cite{Pedro}. 
He shows that isoperimetric regions  are either a cylindrical section or congruent to a ball type region and gives explicit formulae for the volumes 
and areas of the isoperimetric regions and their boundaries. The cylindrical section  $(S^n\times  (a,b)$ has volume $(b-a) V_n$, where
$V_n =Vol(S^n g^n_0 )$, and its boundary has area $2V_n$. Let us recall the values  of $V_n$ that we will use: $V_2 =4\pi$ and
$V_3 = 2 \pi^2$.
 
The ball type regions $\Omega^n_h$ are balls whose boundary 
is a smooth sphere of constant mean curvature $h$. The sections of  $\Omega^n_h$, namely
$\Omega^n_h \cap (S^n \times \{ a \}) $, are geodesic balls in $S^n$ centered at some fixed point. If we let $\eta \in (0,\pi)$
be the maximum of  the radius of those balls then  $h=h_{n-1}(\eta)=\frac{(Sin(\eta))^{n-1}}{\int_0^{\eta}(Sin(s))^{n-1}ds}$.
 The formulas for the volumes of  $\Omega_h$  and its boundary obtained by Pedroza are

			\begin{equation}
					\label{area}
					Vol(\partial \Omega^n_h)= 2  V_{n-1} \int_0^{\eta} \frac{(Sin(y))^{n-1}}{\sqrt{1-u_{n-1}(\eta,y)^2}} dy,
			\end{equation}

			\begin{equation}
				\label{volume}
				Vol(\Omega^n_h)= 2 V_{n-1} \int_0^{\eta} \frac{\int_0^y(Sin(s))^{n-1}ds \  \ u_{n-1}(\eta,y)}{\sqrt{1-u_{n-1}(\eta,y)^2}} dy,
			\end{equation}

\noindent where  \[u_{n-1}(\eta,y)=\frac{(Sin(\eta))^{n-1}/\int_0^{\eta}(Sin(s))^{n-1}ds}{(Sin(y))^{n-1}/\int_0^y(Sin(s))^{n-1}ds}.\]

\noindent Moreover, for $n=2$, one obtains the formulas
		
		\begin{equation}
			\label{area2}
			Vol(\partial \Omega^2_h)=  4 \pi \left( \frac{2}{1+h^2} + \frac{h^2}{(1+h^2)^{3/2}} log\frac{\sqrt{1+h^2}+1}{\sqrt{1+h^2}-1}\right),
		\end{equation}

	\begin{equation}
			\label{volume2}
			Vol(\Omega^2_h)= 4 \pi h \left( \frac{2+h^2}{(1+h^2)^{3/2}}  log\frac{\sqrt{1+h^2}+1}{\sqrt{1+h^2}-1}-\frac{2}{1+h^2}\right),
	\end{equation}

\noindent	
with $h=h_1(\eta)=\frac{Sin(\eta)}{\int_0^{\eta}Sin(s)ds}$.

$Vol(\partial \Omega^2_h)$ is an increasing function of $\eta$ until it reaches the value $8 \pi =2V_2$. This value is achieved for $\eta_0
\approx 1.97$, $h_0 \approx 0.66$. Then for
$\eta \leq \eta_0$ we have $I _{S^2\times \mathbf{R} } (Vol(\Omega^2_h) ) = Vol(\partial \Omega^2_h)$. And for any $v>Vol(\Omega^2_{h_0})$ we have $I(v)=8\pi$ (and
the isoperimetric region is the corresponding spherical cylinder). 
As we mentioned before we picked $\lambda_1$ so that the maximum of $I_{S^3 (3)}$ is $8\pi$. Therefore to make the comparison of the isoperimetric profiles
 we only need to consider the $ \Omega^2_h$ regions and volumes $v \leq Vol(\Omega^2_{h_0})$.

\subsection{Proof that $I_{S^2\times \re ,g_2+dt^2} \geq I_{S^3 (3)}$ .}

The isoperimetric regions in $(S^3,2g^3_0 )$ are geodesic balls, and we have the formula:

$$I_{S^3 (3)}  (2^{5/2} \pi (r-\sin (r)  \cos (r) ) = 8\pi \sin^2 (r) .$$

And as we mentioned above 

$$I _{S^2\times \mathbf{R} } (Vol(\Omega^2_h) ) = Vol(\partial \Omega^2_h) .$$

These formulas are explicit and the only problem to prove the desired inequalities is that one cannot
find the inverse of the functions which give the volumes of the regions. Nevertheless it is very easy 
(and we hope this is clear to the reader) to prove numerically the desired inequality for values of the volume
away from 0.   The problem at 0 is that the isoperimetric profiles are very close and have a singularity at 0
(the derivatives of the functions which give the volume vanish at 0). Since the scalar curvature of $S^2 \times \re^2$
is smaller than $3$ it is know from a result of O. Druet \cite{Druet} that the desired inequality holds for
small values of the volume, but there is no lower bound for how small the volume has to be. We will prove explicitly
that $I_{(S^2\times \re ,g_2+dt^2)} (t) \geq I_{(S^3 (3))}  (t)$ for $t<0.2$ by going through the numerical estimates.
This is of course a very elementary and probably uninteresting job; the reader might want to skip this part and go
directly to the end of this subsection.

In order to prove the required inequality for small values of the volume we need to look at the Taylor expansion
of the formulas for the volumes of the isoperimetric regions and their boundaries. Let $x=1/h$ and call 
$A(x)=Vol(\partial \Omega^2_{1/x})$ and $V(x)=Vol(\Omega^2_{1/x})$. Then we have (by a explicit computation)

$A(x)= 16 \pi x^2 -(64/3)\pi x^4 +(128/5) \pi x^6 -(1024/35) \pi x^8 +o(x^9),$

$V(x)= (32/3) \pi x^3 -(256/15)\pi x^5 +(768/35) \pi x^7 -(8192/315)\pi x^9 + o(x^{10}).$

Moreover, one can easily estimate that for $0<x<0.2$, $A^{(9)}(x)>0$ and $V^{(10)}(x)<3 \times 10^8$ and therefore we have 

$$A(x)>16 \pi x^2 -(64/3)\pi x^4 +(128/5) \pi x^6 -(1024/35) \pi x^8,$$

\noindent and

$$|V(x)-((32/3) \pi x^3 -(256/15)\pi x^5 +(768/35) \pi x^7 -(8192/315) x^9 )|
< 83 x^{10}.$$

The isoperimetric regions in $S^3 (3)$ are geodesic balls, and we have the formula:

$$I_{S^3 (3)}  (2^{5/2} \pi (r-\sin (r)  \cos (r) ) = 8\pi \sin^2 (r) .$$

Then we let 

$$v(r)= 2^{5/2} \pi (r-\sin (r)  \cos (r) )$$

\noindent
and 

$$a(r)= 8\pi \sin^2 (r) .$$

We see 

$$v(r)= 2^{7/2}/3 \pi r^3 -(2^{7/2}/15)\pi r^5 + (16/315) \sqrt{2} \pi r^7 -(8/2835) \sqrt{2} \pi r^9+o(r^{10}),$$

\noindent
and $a(r)= 8 \pi r^2 -(8/3)\pi r^4 +(16/45)\pi r^6 -(8/315)\pi r^8 + o(r^9).$

Now set $r=\sqrt{2}x - (2/5)\sqrt{2}x^3+(11/10)x^5$. Then

$v(x)= \frac{32}{3} \pi x^3 -\frac{256}{15}\pi x^5 +(\frac{22784}{1575}+\frac{88}{5} \sqrt{2}) \pi x^7-(\frac{699904}{70875}+
\frac{1936}{75} \sqrt{2})\pi x^9+ o(x^{10})$

$a(x)= 16 \pi x^2 -\frac{352}{15}\pi x^4 + (\frac{5056}{225}+\frac{88}{5} \sqrt{2})\pi x^6- (\frac{5504}{315}+\frac{2288}{75} \sqrt{2}) \pi x^8  +o(x^9) $.

Moreover, one can easily check that for $0<x<0.2$, $a^{(9)}(x)<7 \times 10^7$ and $v^{(10)}(x)>0$. Therefore we have 

\noindent
$|a(x)-(16 \pi x^2 -\frac{352}{15}\pi x^4 + ( \frac{5056}{225}+\frac{88}{5} \sqrt{2} ) \pi x^6- (\frac{5504}{315}+\frac{2288}{75} \sqrt{2}) \pi x^8 )|<7 \frac{10^7 }{9!}  x^9 < 193 x^9$

\noindent
 and 
 
 \noindent
 $v(x)> \frac{32}{3} \pi x^3 -\frac{256}{15}\pi x^5 +(\frac{22784}{1575}+\frac{88}{5} \sqrt{2}) \pi x^7-(\frac{699904}{70875}+\frac{1936}{75} \sqrt{2})\pi x^9 .$

It follows that for $0<x<0.2$ we have 

$$A(x)-a(x)> \frac{32}{15} \pi x^4 +\left( \frac{704}{225}-\frac{88}{5} \sqrt{2} \right) \pi x^6+\left( -\frac{3712}{315}+\frac{2288}{75} \sqrt{2}\right) \pi x^8- 193 x^9$$
$$>\left(\frac{32}{15} \pi +\left( \frac{704}{225}-\frac{88}{5} \sqrt{2}\right) \pi (0.2)^2- 193 (0.2)^5\right) x^4 >3 x^4>0,$$
and

$$v(x)-V(x)>\left( -\frac{11776}{1575}+ \frac{88}{5} \sqrt{2}\right) \pi x^7+\left( \frac{163328}{10125}-\frac{1936}{75} \sqrt{2}\right) \pi x^9 - 83 x^{10}$$
$$>\left( \left( \frac{-11776}{1575}+ \frac{88}{5}\sqrt{2}\right) \pi +\left( \frac{163328}{10125}-\frac{1936}{75} \sqrt{2}\right) \pi (0.2)^2 - 83 (0.2)^3 \right)x^7$$
$$>51 x^7>0.$$

Now, for $0<x<0.2$ we have $I_{S^2\times \re}\left(V(x)\right)= A(x)>a(x)=I_{S^3(3)}\left(v(x)\right)> I_{S^3(3)}\left(V(x)\right)$ (since $I_{S^3(3)}$ is increasing). And so, we have that if $0<x<0.2$ then  $I_{S^2\times \re}\left(V(x)\right)> I_{S^3(3)}\left(V(x)\right)$. Since $V(0.2)>0.25$ we have proved that 
$I_{S^2\times \re} (t)> I_{S^3(3)} (t)$ for $t<0.25$.

Once we dealt with sufficiently small values of the volume checking the desired inequality for the isoperimetric functions is a completely standard numerical argument. Instead of going through it, we will simply provide with the graphics for  $x > 0.5$  (Figure \ref{fig:S2}a), for $0.3 < x < 0.5$ (Figure \ref{fig:S2}b) and for  $0.2 < x<0.3$ (Figure \ref{fig:S2}c).

\begin{figure}[h!bp]
\begin{center}
\subfigure[$x > 0.5$.]{
\includegraphics[scale=0.30]{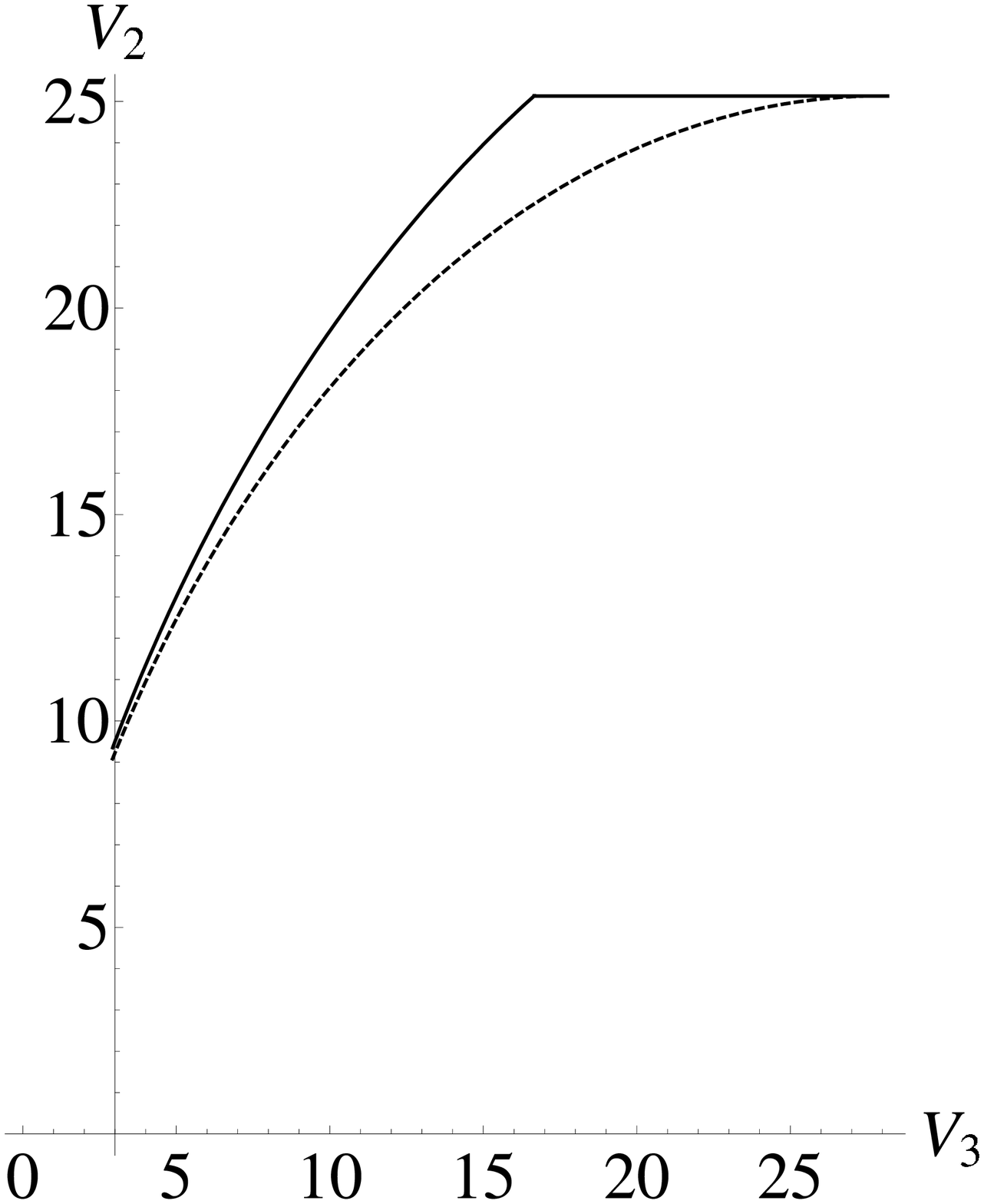}}
\subfigure[$0.3 < x < 0.5$ .]{
\includegraphics[scale=0.30]{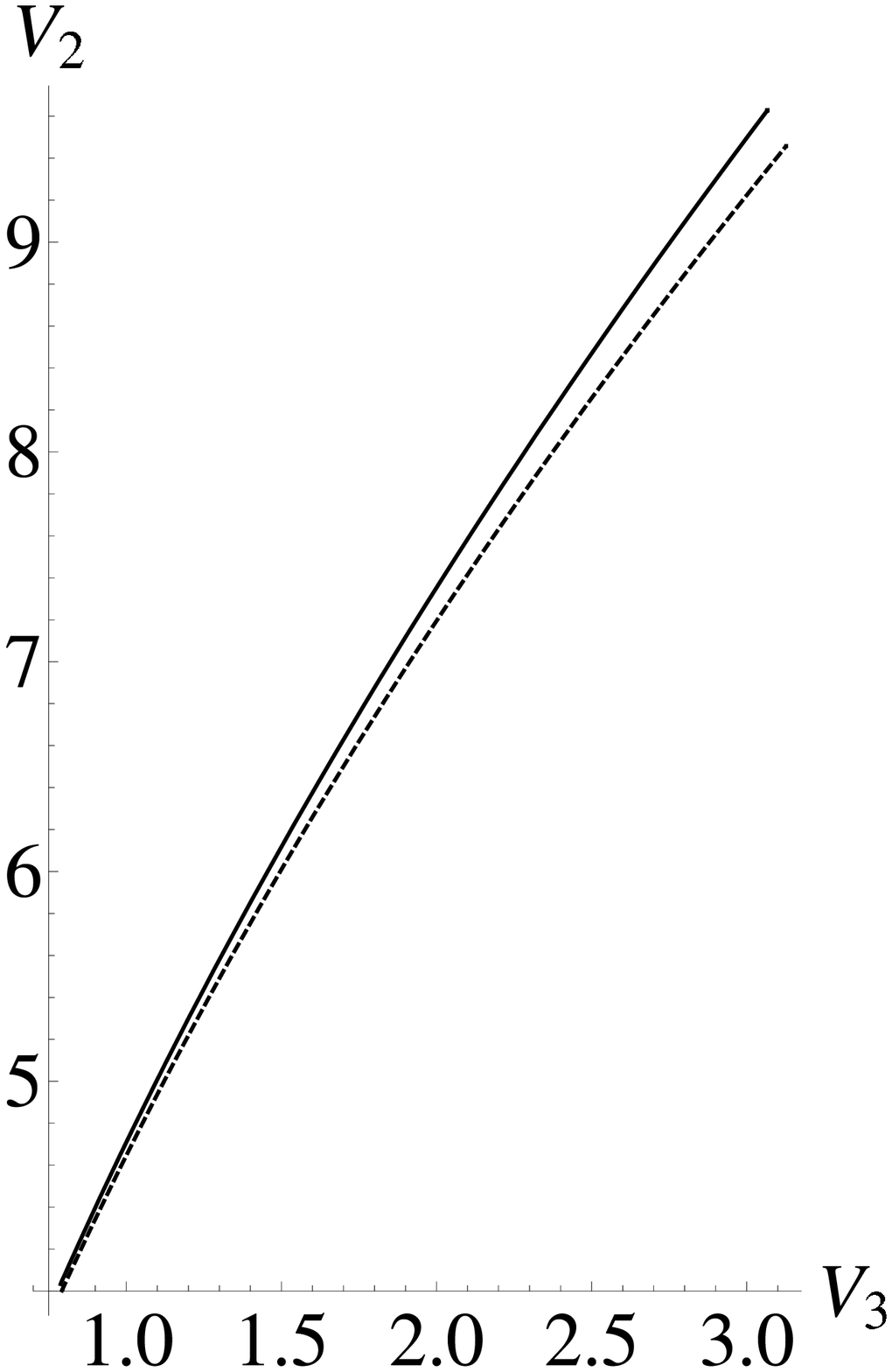}}
\subfigure[$0.2 < x<0.3$.]{
\includegraphics[scale=0.30]{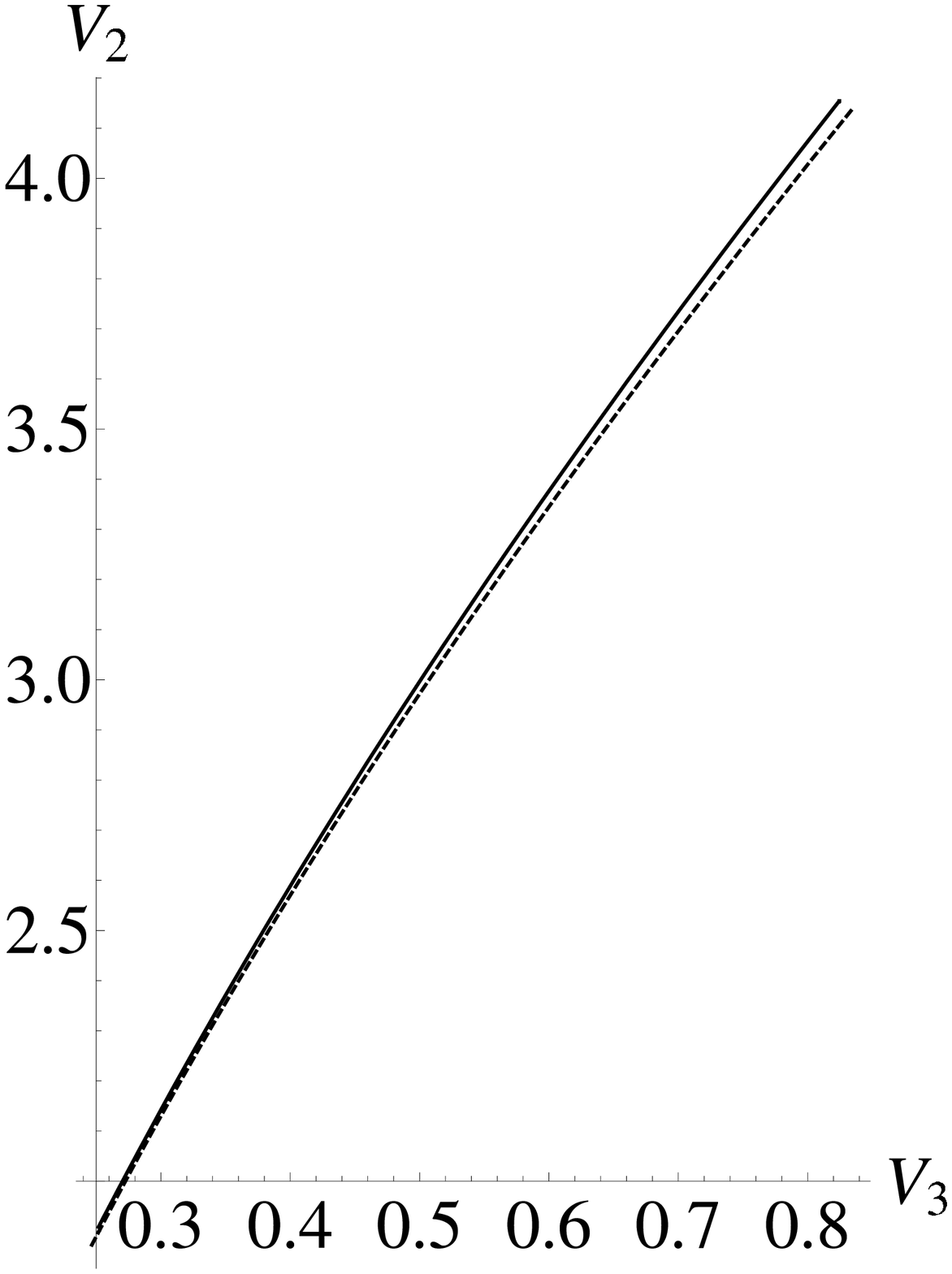}}
\caption{Comparison of the isoperimetric profiles $I_{S^2 \times \re}$ and $I_{S^3(3)}$.}
\label{fig:S2}
\end{center}
\end{figure}

\subsection{Proof that $I_{S^3 (3) \times \re }(t)\geq    {\frac{2\sqrt{\epsilon}}{12^{3/8}}} I_{S^4(\frac{ 12^{3/4}}{2\epsilon})}(t)$ for $t\leq 100$, $\epsilon =(1.047)^2$.} 
The situation is similar to the one in the previous subsection. $S^3 (3) \times \re$ is isometric to $(S^3 \times \re , 2(g_0^3 +dt^2 ) )$.  Then we consider the formulas
(1) and (2) and let $x=\eta$. Then if we call $V(x)=Vol(\Omega^3_{h(x)} )$ and $A(x)=Vol(\partial \Omega^3_{h(x)} )$ it follows 
that $I_{S^3 (3) \times \re } (4V(x) )= 2^{3/2} A(x)$ for small values of $x$. This holds until $x=x_0 \approx 1.9$ when $2^{3/2} A(x_0 ) =8\sqrt{2} \pi^2
=2Vol(S^3 (3))$. 
Let $v_0 =4V(x_0 ) $. Then
for $v\geq v_0 $
we have that $I_{S^3 (3) \times \re } (v)=8\sqrt{2} \pi^2$. The only problem to verify the inequality is for small values of the volumes. In this case
the problem becomes simpler because of the factor ${\frac{2\sqrt{\epsilon}}{12^{3/8}}} \approx 0.825 <1$.

By a direct computation we see that $4V(1)<15$ and $2^{3/2} A(1)>39$. It follows that $I_{S^3 (3) \times \re } (15)/15^{3/4} >39/15^{3/4} >5$.

But it was proved by V. Bayle \cite[Page 52]{Bayle} that the function $I_{S^3 (3) \times \re } (v)/v^{3/4}$ is decreasing. So for any $v<15$ we
have that $I_{S^3 (3) \times \re } (v) > 5 v^{3/4}$

Of course, $S^4(\frac{ 12^{3/4}}{2\epsilon}) = (S^4, 2^{3/2} 3^{1/4} \epsilon g_0^4 )$ and so the isoperimetric profile is given by 

$$I_{S^4(\frac{ 12^{3/4}}{2\epsilon})}\left(\epsilon^2  \frac{64 \pi^2 }{\sqrt{3}} (2+\cos(r))\sin^4(r/2) \right)=8 \ \ 2^{1/4}3^{3/8}  \pi^2 \epsilon^{3/2} \sin^3 (r).$$

Let us call $I_1 = I_{S^4(\frac{ 12^{3/4}}{2\epsilon})}$. Then one can trivially check that $\lim_{v\rightarrow 0} I_1 (v)/v^{3/4} = 2^{7/4} \sqrt{\pi} <6$. Since
the function  $I_1 (v)/v^{3/4}$ is decreasing by the Theorem of  Bayle  we have that 

$$ I_{S^4(\frac{ 12^{3/4}}{2\epsilon})} (v)  < 6v^{3/4}.$$

And so for $v<15$, $I_{S^3 (3) \times \re } (v) > 5v^{3/4}>(0.83) 6 v^{3/4}> {\frac{2\sqrt{\epsilon}}{12^{3/8}}} I_{S^4(\frac{ 12^{3/4}}{2\epsilon})} (v) $.

 For bigger values ($15\leq v \leq 100$), we simply check the inequality of the isoperimetric functions through a completely standard numerical argument.  We provide with the graphic for  $5 \leq t \leq 100$ (Figure \ref{fig:S3}).

\begin{figure}[h!tp]
\begin{center}
\includegraphics[scale=0.30]{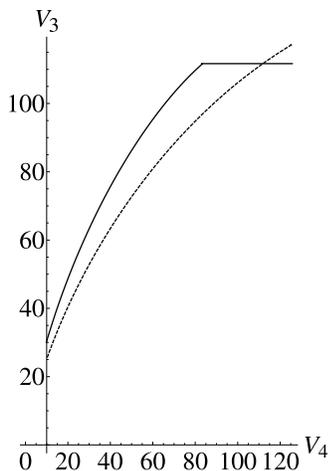}
\caption{\label{fig:S3}Comparison of the isoperimetric profiles $I_{S^3(3) \times \re}$ (continuous) and $\sqrt{\frac{\sqrt{2} \epsilon}{3^{3/4}} } I_{S^4(\frac{ 12^{3/4}}{2\epsilon})}$ (dashed).}
\end{center}
\end{figure}

\subsection{Proof of Theorem 1.1}

\begin{Lemma} $I_{(S^2 \times \re^2 ,g_0^2 +dt^2)} \geq I_{S^3 (3)\times \re}$ 
\end{Lemma}

\begin{proof}
Consider an isoperimetric region $U \subset S^2 \times \re^2$. Consider any $t \in \re$ and let $U_t =
U\cap (S^2 \times \re \times \{ t \}) $. If $Vol(U_t ) \leq Vol ( S^3 (3) )$ we let $W_t$ be the geodesic sphere 
in $S^3 (3)$ around the south pole with volume $Vol(U_t )$. If $Vol(U_t ) > Vol (S^3 (3) )$ we let $W_t =
S^3 (3)$.

If $Vol(U_t ) \leq Vol S^3 (3)$ for all $t$ then
we consider the region $W\subset S^3 (3) \times \re$ such that $W\cap (S^3 \times \{ t \}) = W_t$. Then $Vol(W)=
Vol(U)$. But since $I(S^2\times \re ,g_2+dt^2) \geq I(S^3 (3))$ from Subsection 1 we can apply
Ros Product Theorem (see \cite[Proposition 3.6]{Ros} or \cite[Section 3]{Morgan2}) to see
that  $Vol (\partial W) \leq Vol(\partial U)$. 

In case $Vol(U_t ) > Vol( S^3 (3))$ for some $t$ then there is  some interval $(-a,a)$ where this happens,
and  $Vol(U_a ) = Vol( S^3 (3))$. In the same way as before we construct a region $W
\subset S^3 (3) \times \re$. As before $Vol(\partial W) \leq Vol(\partial U)$ by Subsection 1 and
Ros Product Theorem, but in this case $Vol(W)<Vol(U)$. But we can replace $W$ by a region
$\overline{W}$ which has a thicker cylindrical part (a  region $(-a,B)\times S^3(3)$ with $B>A$)
such that $Vol(\partial \overline{W} ) =Vol(\partial W )$
and $Vol( \overline{W}) = Vol(U)$.  

This proves the Lemma.

\end{proof}

\begin{Lemma} 
\label{MorganLemma}
The isoperimetric profile of $S^2 \times \re^2$, $I_{S^2 \times \re^2}(v)$, is bounded below by $\frac{4 \pi}{\sqrt{2}} \sqrt{v}$, for $v\geq 16$.
\end{Lemma}

\begin{proof}
Consider the isoperimetric profile of  $\re^2$, $f_1(v)=2\sqrt{\pi} \sqrt{v}$, and the isoperimetric profile of $S^2$, $f_2(v)=\sqrt{v(4 \pi-v)}$ ($f_2$ is defined on $[0,4\pi]$). Let

$$I_P(v)= \inf \{v_1 f_2(v_2) + v_2 f_1(v_1) | v_1v_2 = v \}$$

\noindent be the lower bound on the isoperimetric profile of $S^2 \times \re^2$ for regions which are products. Then, since $f_1$ and $f_2$ are concave, it follows by Theorem 2.1 in \cite{Morgan1} that

\begin{equation}
\label{Mor}
I_{S^2 \times \re^2}(v)\geq \frac{1}{\sqrt{2}} I_P(v).
\end{equation}

To compute $I_P (v)$, let $x=v_2 \in (0,4\pi )$ and  $v_1 =v/x$. Consider

$$f_v (x)=(v/x)\sqrt{x(4\pi -x)} + 2\sqrt{\pi} x \sqrt{v/x} = (v/\sqrt{x}) \sqrt{4\pi -x}  +2\sqrt{\pi} \sqrt{x} \sqrt{v} .$$

One can find the infimum of $f_v$ explicitly, but it is a little cumbersome. For our purposes it is enough
to consider the case $v\geq 16$. Then

$$f_v  (x) \geq (4  \sqrt{(4\pi /x) -1} +2 \sqrt{\pi } \sqrt{x} )\sqrt{v}.$$

But it is easy to check that for $x\in (0,4 \pi )$, $4  \sqrt{(4\pi /x) -1} +2 \sqrt{\pi } \sqrt{x} \geq 4 \pi $. Then
$I_P (v) \geq 4 \pi \sqrt{v}$ for $v\geq 16$ and the lemma follows.

\end{proof}


We are now ready to give the proof of Theorem 1.1:



\begin{proof}
It follows from Subsection 2.2 and Lemma 2.1 that the theorem holds for $v\leq 100$. 
On the other hand, since the Ricci curvature of $S^2 \times \re^2$ is non-negative, it follows from Corollary 2.2.8 in \cite{Bayle}, page 52, that the isoperimetric profile  of $S^2 \times \re^2$ is concave. In turn, this concavity of $I_{S^2 \times \re^2}(v)$ implies that $I_{S^2 \times \re^2}(v)$ is also bounded from below by $l(v)$ (for $v_1\leq v \leq v_2$,  $v_2\geq 16)$, where $l(v)$ is the straight line joining the two points $\left(v_1,I_{S^3(3) \times \re}(v_1)\right)$ and $\left(v_2,\frac{4 \pi}{\sqrt{2}} \sqrt{v_2}\right)$, of the graphs of $ I_{S^3(3) \times \re}(v)$ and  $\frac{4 \pi}{\sqrt{2}} \sqrt{v}$. In particular by choosing $v_1=83.5$ and $v_2=450$, we get the line $l(v)=0.209642 (v-83.5)+(8 \sqrt{2} \pi^2)$, as a lower bound for $I_{S^2 \times \re^2}(v)$, for $83.5 \leq v \leq 450$. Again, using standard numerical computations, we show that this line $l(v)$, in turn, bounds  $\sqrt{\frac{\sqrt{2} \epsilon}{3^{3/4}} }I_{S^4(\frac{ 12^{3/4}}{2\epsilon})}(v)$ from above, for $v \geq 83.5$, Figure \ref{fig:line}. We provide the graphics.

\begin{figure}[htp]
\subfigure[The line $l(v)$ (dashed),
 joining  $I_{S^3(3)\times \re}\left(v \right)$ and 
  $\frac{4 \pi}{\sqrt{2}} \sqrt{v}$,  is a lower bound 
 for   $I_{S^2 \times \re^2}(v)$,
 for $83.5\leq v \leq 450$.]{
\includegraphics[scale=0.30]{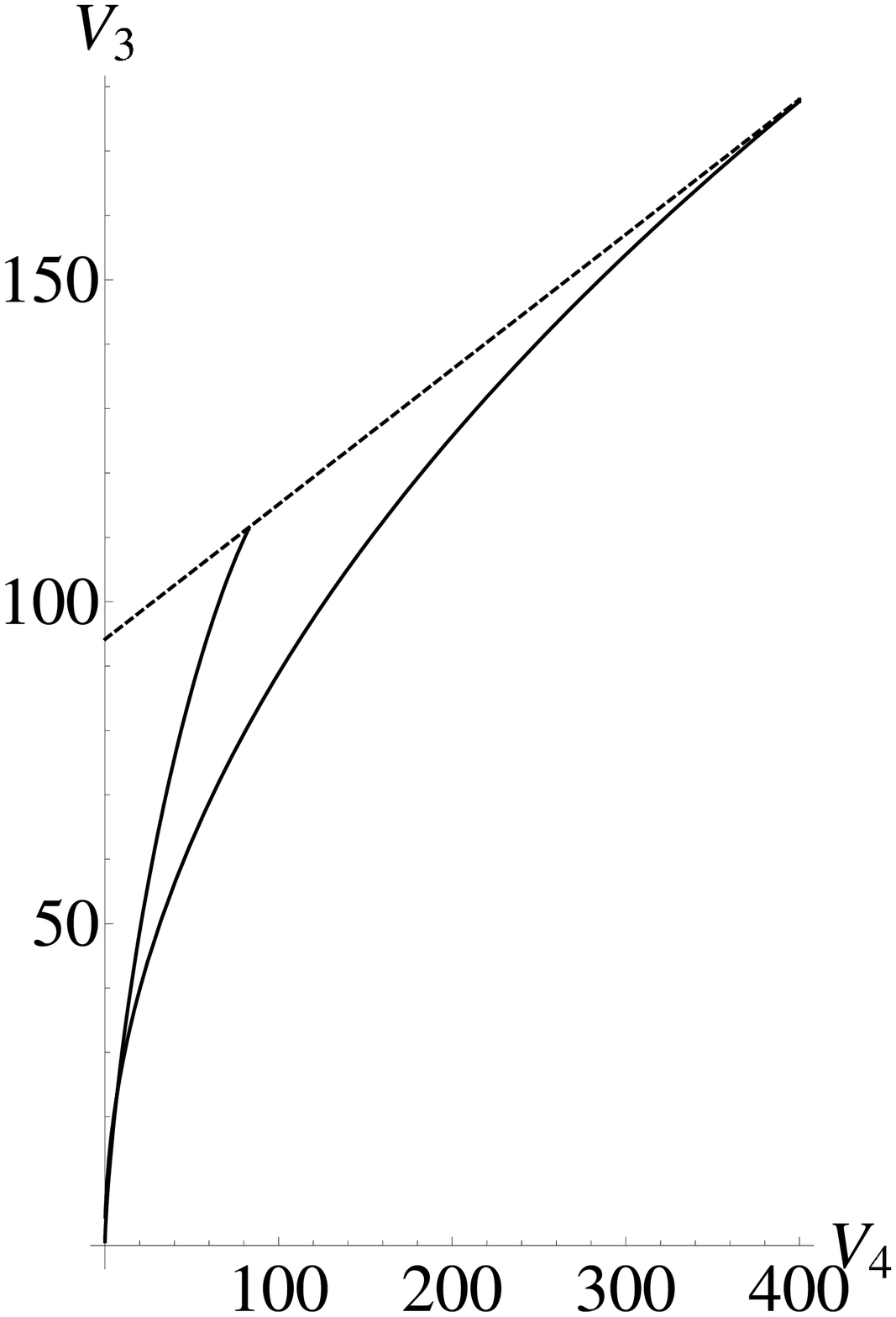}}
\hspace{0.3in}%
\subfigure[The line $l(v)$ (dashed) is an upper bound 
for $\sqrt{\frac{\sqrt{2} \epsilon}{3^{3/4}} }I_{S^4(\frac{ 12^{3/4}}{2\epsilon})}(v)$ (continuous).]{
\includegraphics[scale=0.30]{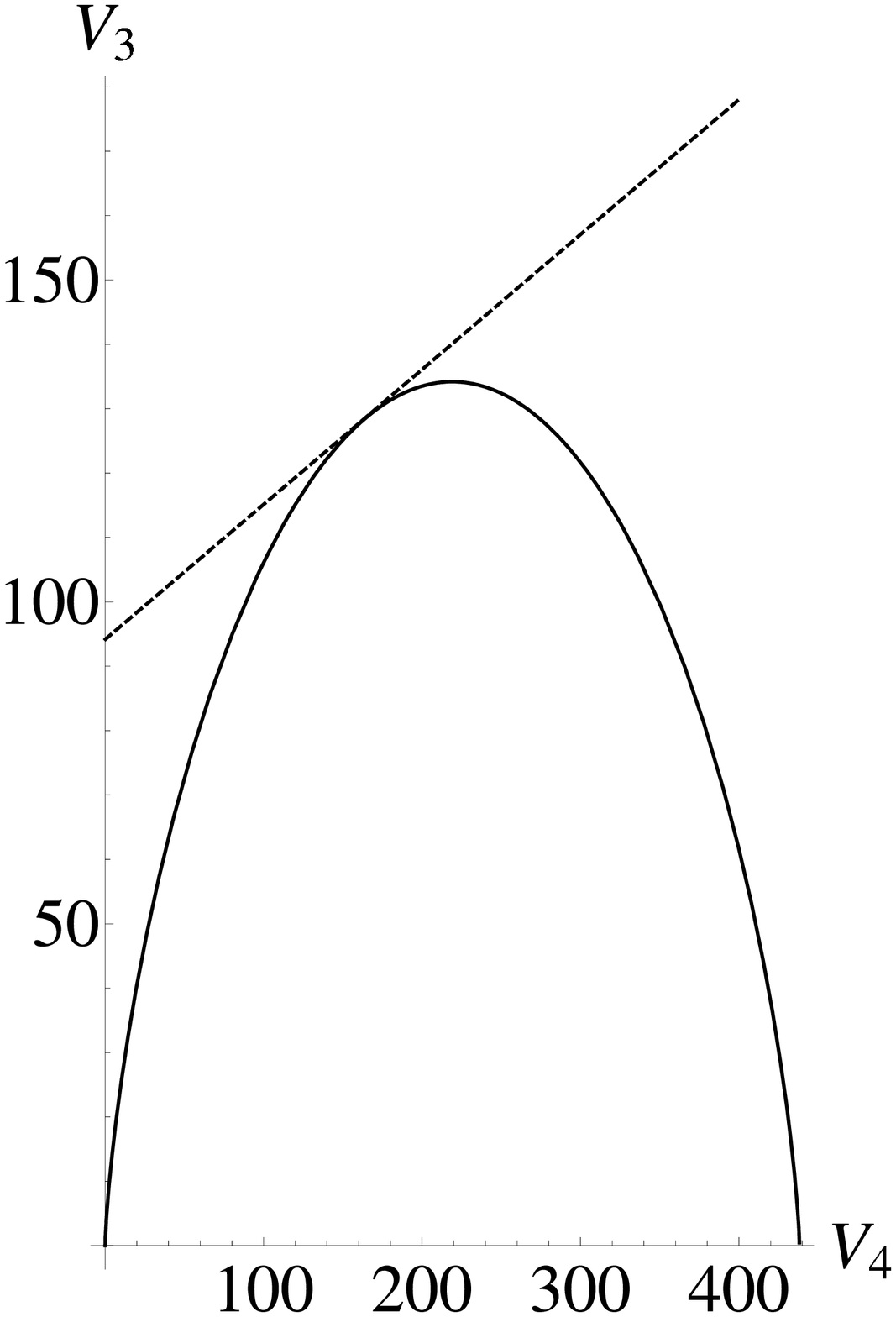}}

\caption{\label{fig:line} The line $l(v)$ is a lower bound for $I_{S^2 \times \re^2}$, and an upper bound for $\sqrt{\frac{\sqrt{2} \epsilon}{3^{3/4}} }I_{S^4(\frac{ 12^{3/4}}{2\epsilon})}(v)$, for $83.5\leq v \leq 450$.}
\end{figure}

\end{proof}


\section{Proof of Theorem 1.2}

\begin{proof}
Let $f: S^2 \times \re^2 \rightarrow \re_{\geq 0}$ be any smooth compactly supported
function.  Recall that we denote
by $g^n_0$ the metric on $S^n$ of constant sectional curvature one and by $g^n_0 (k)$ the
round metric on $S^n$ of scalar curvature $k$. So $g^n_0 (k) = n(n-1)/k \ g^n_0$. By
$S^n (k)$ we mean the Riemannian manifold $(S^n ,g^n_0 (k))$.

In case $Vol(\{ f>0 \} )\leq Vol(S^4 (12^{3/4}/(2\epsilon )) ) $ we let $f_* : S^4 (12^{3/4}/(2\epsilon ) )  \rightarrow
\re_{\geq 0}$ be the spherical symmetrization of $f$: $f_*$ is a radial (with
respect to the axis through some fixed point $S$), non-increasing function on
the sphere such that for any $t>0$, $Vol(\{ f>t \} )=Vol(\{ f_* >t \} )$. Then for any $q>0$, 
$||f||_q = ||f_* ||_q$.

Now, by the coarea formula 

$$ \int { \| \nabla f \| }^2  dvol(g^2_0 +dt^2 )  = \int_0^{\infty} \left( \int_{f^{-1} (t)} \| \nabla f \| d\sigma_t  \right) dt ,$$

\noindent
where $d\sigma_t$ denotes the volume element of the induced metric on $f^{-1} (t) $.
By H\"{o}lder's inequality

$$\int_0^{\infty} \left( \int_{f^{-1} (t)} \| \nabla f \| d\sigma_t  \right) dt  \geq \int_0^{\infty} (Vol(f^{-1} (t) ))^2 
{\left( \int_{f^{-1} (t) } {\| \nabla  f \|}^{-1} d\sigma_t \right) }^{-1} dt . $$

But 

$$ \int_{f^{-1} (t) } {\| \nabla  f \|}^{-1} d\sigma_t  = -\frac{d}{dt}(\{ f>t \} ) =-\frac{d}{dt}(Vol(\{ f_* >t \} ))=
 \int_{f_*^{-1} (t) } {\| \nabla  f_*  \|}^{-1} d\sigma_t .$$

Now $f^{-1} (t) $ contains  the boundary of $\{ f>t \} $ and $Vol( \{ f>t \} ) =Vol(\{ f_* >t \} )$ (which is an isoperimetric
region in the sphere). Then Theorem 1.1 tells us that $ Vol(f^{-1} (t) ) \geq 
Vol(\partial (\{ f>t \} ) \geq   \frac{2\sqrt{\epsilon}}{12^{3/8}}   Vol( f_*^{-1} (t) )$, and so

$$\int \| \nabla f \|^2  dvol(g^2_0 +dt^2 ) \geq  {\left( \frac{2\sqrt{\epsilon}}{12^{3/8}} \right)}^2  \int_0^{\infty} (Vol(f_*^{-1} (t) ))^2 
{\left( \int_{f_*^{-1} (t) } {\| \nabla  f_* \|}^{-1} d\sigma_t \right) }^{-1} dt $$

$$= (\sqrt{2}\epsilon / 3^{3/4} ) \int_0^{\infty} Vol(f_*^{-1} (t)) \| \nabla f_* \| dt =  (\sqrt{2}\epsilon / 3^{3/4}) \int_0^{\infty} \left( \int_{f_*^{-1} (t)} \| \nabla f_*  \| d\sigma_t  \right) dt $$

$$= (\sqrt{2} \epsilon / 3^{3/4})  \int \| \nabla f_*  \|^2 dvol(g^4_0 ( 12^{3/4}/(2\epsilon ) ) )$$

\noindent
(we are using that $\| \nabla f_*  \|$ is constant along level surfaces of $f_*$, since it is a radial function).
It follows that

$$Y_{g^2_0 + dt^2} (f) = \frac{6 \int_{S^2 \times \re^2} {\| \nabla f \|}^2\  dvol(g^2_0 +dt^2 ) + \int_{S^2 \times \re^2}  2 f^2 \ dvol(g^2_0 +dt^2 ) }{ (\int_{S^2 \times \re^2}
 f^4 \ dvol(g^2_0 +dt^2))^{1/2} }$$

$$\geq \frac{6 (\sqrt{2}\epsilon / 3^{3/4}) \int_{S^4} {\| \nabla f _* \| }^2  \ dvol(g^4_0 (12^{3/4}/(2\epsilon ) ) ) + \int_{S^4} 2 f_*^2 \ dvol(g^4_0 (12^{3/4}/(2\epsilon ))) }{ (\int_{S^4} f_*^4 \ dvol(g^4_0(12^{3/4}/(2\epsilon ))))^{1/2} } $$

$$\geq \frac{\sqrt{2} \epsilon}{3^{3/4}} \ \  \frac{6 \int_{S^4} {\| \nabla f _* \| }^2 \  dvol(g^4_0 ( 12^{3/4}/(2\epsilon )  )) + \int_{S^4}   (12^{3/4} /(2\epsilon ))     f_*^2 \ 
 dvol(g^4_0 ( 12^{3/4}/(2\epsilon ) )) }{ (\int_{S^4}
 f_*^4 \ dvol(g^4_0 ( 12^{3/4}/(2\epsilon )  )))^{1/2} } $$

$$= \frac{\sqrt{2} \epsilon }{3^{3/4}} Y_{g^4_0 ( 12^{3/4} /(2\epsilon ))} (f_* ) .$$

\vspace{.5cm}

Now if $Vol(\{ f>0 \} ) >  Vol(S^4 (  12^{3/4} /(2\epsilon )))$ there are values $t_0 = \max (f) \geq t_1
\geq t_2 \geq ...\geq t_N =0$ such that $Vol (f^{-1} (t_{i} ,t_{i-1} ))= Vol(S^4 (   12^{3/4}/(2\epsilon )))$,
for $i=1,..,N-1$ and $Vol(f^{-1}(0,t_{N-1})) \leq  Vol(S^4 ( 12^{3/4}/(2\epsilon ) ))$.  Let $f_i$
be the restriction of $f$ to $f^{-1} (t_{i} ,t_{i-1}  )$, for $i=1,2,..,N$, and let ${f_i}_* $ be the spherical
symmetrization of $f_{i}$. Therefore   ${f_i}_*  : S^4 (    12^{3/4}/(2\epsilon )) \rightarrow 
[t_{i} , t_{i-1} ]$ is radial (with respect to some chosen point $S$, the south pole),
non-increasing and $Vol(\{ f_i >t \} )= Vol( \{  {f_i}_* >t \} )$ (for all $t>0$). Note that
according to our notations for $t\in  [t_{i} , t_{i-1} ]$, we have  

$$Vol(\{ { f_i }_* >t \}  ) = Vol(\{ f_i >t \}  ) = Vol (f^{-1} (t,t_i )) \leq
Vol(\{ f> t \} ) .$$

By a result of V. Bayle \cite[Page 52]{Bayle} the isoperimetric profile
$I_{S^2 \times \re^2}$ is concave and therefore increasing. Then it follows from Theorem 1.1   that for any $t>0$ we have that
$Vol({f_i } ^{-1} (t) ) \geq   \frac{2\sqrt{\epsilon}}{12^{3/8}}Vol({{f _i}_*}^{-1} (t))$ . As before we apply
the coarea formula to obtain

$$ \int_{ f^{-1} (t_{i} ,t_{i-1}  ) } {\| \nabla f_i \| }^2 dvol(g^2_0 +dt^2 ) \geq {\left( \frac{2\sqrt{\epsilon}}{12^{3/8}}
\right) }^2  \int_{S^4}  {\| \nabla {f_i}_* \| }^2 dvol(g^4_0 (12^{3/4}/(2\epsilon ))) .$$

\noindent 
Therefore we  have that for any $q>0$, 

$$|| f ||_q^q = \Sigma_{i=1}^{N} || f _i ||_q^q = \Sigma_{i=1}^{N} || {f _i}_* ||_q^q ,$$

\noindent
and

$$ \int_{S^2 \times \re^2} {\| \nabla f \| }^2  dvol(g^2_0 +dt^2 )\geq     {\left( \frac{2\sqrt{\epsilon}}{12^{3/8}}
\right) }^2  \   \Sigma_{i=1}^{N} \int_{S^4}  {\| \nabla {f_i}_* \| }^2 
dvol(g^4_0 (12^{3/4} /(2\epsilon ))) .$$

Then 

$$Y_{g^2_0 + dt^2} (f) =  \frac{ \int_{S^2 \times \re^2} 6   {\| \nabla f \| }^2 +2 f^2  \ dvol (g^2_0 +dt^2 )}{ (\int_{S^2 \times \re^2} f^4 \   dvol(g^2_0 +dt^2) \ )^{1/2} } $$

$$\geq  \frac{  \Sigma_{i=1}^{N}     {\left( \frac{2\sqrt{\epsilon}}{12^{3/8}}
\right) }^2  \int_{S^4} 6  {\| \nabla {f_i}_* \| }^2  + 2 {f_i}_*^2  \  dvol (g^4_0 (  12^{3/4}/(2\epsilon ) ) )}{ (\Sigma_{i=1}^{N} \int {f_i}_*^4)^{1/2} }  $$

$$=  {\left( \frac{2\sqrt{\epsilon}}{12^{3/8}}
\right)}^2 \ \ 
\frac{  \Sigma_{i=1}^{N}    \left(  \int_{S^4} 6  {\| \nabla {f_i}_* \| }^2  \ + \  ( 12^{3/4}/(2\epsilon )     ) \  {f_i}_*^2 \ \  dvol(g_0 (12^{3/4}/(2\epsilon)) ) \  \right) }{( \Sigma_{i=1}^{N}   \int{ f_i}_*^4)^{1/2} }  $$

\noindent
and since for any $i$

$$ \int_{S^4}  6 {\| \nabla {f_i}_* \| }^2  +     (12^{3/4}/(2\epsilon ))     {f_i}_*^2 \ dvol(g_0 ( 12^{3/4}/(2\epsilon )    ) ) \geq Y_4  {\left( \int_{S^4}  {f_i}_*^4 dvol(g^4_0 ( 12^{3/4} /(2\epsilon ) ) \right) }^{1/2} ,$$

\noindent
we have

$$Y_{g^2_0 + dt^2} (f)  \geq          \frac{\sqrt{2} \epsilon }{3^{3/4}}        \ Y_4 \ \frac{  \Sigma_{i=1}^{N}  ( \int_{S^4}  {{f_i}_* }^4 )^{1/2} }{(\Sigma_{i=1}^{N}   \int _{S^4} {f_i}_*^4 )^{1/2} }  \geq 
\frac{\sqrt{2} \epsilon }{3^{3/4}}             Y_4
.$$

And therefore

$$Y(S^2 \times \re^2 , [g^2_0 +dt^2 ])  = \inf_{f} Y_{g^2_0 + dt^2} (f) \geq       \frac{\sqrt{2} \epsilon }{3^{3/4}}        \  Y_4 .$$

\end{proof}

\vspace{0.5cm}

\end{document}